\documentclass[a4paper]{amsart}
\usepackage{amssymb}
\usepackage{url} 

\theoremstyle{plain}
\newtheorem{thm}[subsection]{Theorem}
\newtheorem{prop}[subsection]{Proposition}
\newtheorem{lem}[subsection]{Lemma}
\newtheorem{cor}[subsection]{Corollary}

\theoremstyle{definition} 
\newtheorem{defn}[subsection]{Definition}
\newtheorem{rem}[subsection]{Remark}

\newtheorem*{ack}{Acknowledgments}

\begin{document}

%%%%%%%%%%%%%%%%%%%%%%%%%
% Subject classification 
%%%%%%%%%%%%%%%%%%%%%%%%%

% Provide an AMS subject classification with one or two primary classification 
% numbers and, optionally, one or more secondary classification numbers. 
% Use the following format:  "Primary 42B25. Secondary 42B60, 20E26"


%%%%%%%%%
% Title
%%%%%%%%%

% Title, in lower case, with no explicit linebreaks (\\).  If the title
% is too long to be used as a running head, add a short version of the
% title in brackets, as in \title[shorttitle]{fulltitle}.

\title[The core of a module and the adjoint of an ideal]{The core of a module and the adjoint of an ideal over a two dimensional regular local ring}

%%%%%%%%%%%%%%%%%%%%%%%%%%%%%%
% Author names and addresses 
%%%%%%%%%%%%%%%%%%%%%%%%%%%%%%

% Provide one separate \author{...} \address{...} \email{....} entry for each
% author, i.e., do not combine multiple authors.  Separate address lines by double
% slashes.  Do not attach footnotes to author  names. (For acknowledgements use
% the "\thanks" construct below.)
%

\author{Kohsuke Shibata}

\address{Graduate School of Mathematical Sciences, University of Tokyo, 3-8-1 %\newline
Komaba, Meguro-ku, 
Tokyo, 153-8914, Japan.}

\email{shibata@ms.u-tokyo.ac.jp}

%%%%%%%%%%%%%%%%%%%%
% Acknowledgements
%%%%%%%%%%%%%%%%%%%

% Use \thanks for acknowledgements as footnotes to the title page.  
% (Note that footnotes inside \author or \title macros are not
% allowed.)
%
% In case of multiple author papers, phrase the acknowledgement to 
% say "The first author was supported by ...  The second author was
% supported by ..."

\thanks{
The author was supported by JSPS Grant-in-Aid for Early-Career Scientists 19K14496.
}

%%%%%%%%%%%%%
% Abstract 
%%%%%%%%%%%%%
%
% Abstracts should not contain macros (so that they can be processed independently
% of the paper.) Avoid displayed math and references in the abstract.

\begin{abstract}
The core of an module is the intersection of all its reductions.
The main result asserts that the core of a finitely generated, torsion-free, integrally closed module over a two dimensional regular local ring is the product of the module and the adjoint of an ideal.
 This generalizes the fundamental formula for the core of an integrally closed ideal in a two-dimensional regular local ring due to Huneke and Swanson.
As an application, we show that
for integrally closed modules $M$ and $N$  over a two-dimensional regular local ring with $M\subset N$ and $M^{**}=N^{**}$,
 the core of  $M$ is contained in the core of $N$.
\end{abstract}

\maketitle

%%%%%%%%%%%%%%%%%%%%%%%%%%%%%%%%%%%%%%%%%%%%%%%%%%%%%%%%%%%%%%%%%%%%%%%%%
% end Topmatter
%%%%%%%%%%%%%%%%%%%%%%%%%%%%%%%%%%%%%%%%%%%%%%%%%%%%%%%%%%%%%%%%%%%%%%%%%

%%%%%%%%%%%%%%%%%%%%%%%%%%%%%%%%%%%%%%%%%%%%%%%%%%%%%%%%%%%%%%%%%%%%%%%%%
% body of paper
%%%%%%%%%%%%%%%%%%%%%%%%%%%%%%%%%%%%%%%%%%%%%%%%%%%%%%%%%%%%%%%%%%%%%%%%%

\section{Introduction}

Let $I$ be an ideal of a Noetherian ring.
An ideal $J\subset I$ is called a reduction of $I$ if there exists a positive integer $n$ such that $JI^{n}=I^{n+1}$.
The core of $I$, denoted by $\mathrm{core}(I)$, is defined to be the intersection of all reductions of $I$.
The core of an ideal was  introduced   by Rees and Sally in \cite{RS}.
Huneke and Swanson showed many properties of the core of an ideal of a $2$-dimensional regular local ring and the various relationships between the core of an ideal and the adjoint of an ideal  of a $2$-dimensional regular local ring  in \cite{HS}.
Recently, their  results were generalized to a $2$-dimensional  local ring with rational singularity by Okuma, Watanabe and Yoshida (\cite{OWY}) and the author (\cite{S}).

This paper generalizes the results in \cite{HS} in a different direction.
The notions of integral closures and reductions of finitely generated torsion-free modules were introduced by Rees in \cite{R}.
The core of a module $M$, denoted by $\mathrm{core}(M)$, is defined to be the intersection of all reductions of $M$ in the same way as for ideals.
Then it is natural to ask whether the results in \cite{HS} can be generalized to the core of a module.
In \cite{HS} Huneke and Swanson proved that  the core of an integrally closed ideal in a two-dimensional regular local ring is a product of the ideal and a certain Fitting ideal of the ideal.
This result was generalized to integrally closed modules by Mohan (\cite{M}).

In \cite{HS} Huneke and Swanson also proved that
$\mathrm{core}(I)=\mathrm{adj}(I)I$ for an  $\mathfrak m$-primary integrally closed  ideal $I$ of a $2$-dimensional regular local ring $(R,\mathfrak m)$ with infinite residue field.
Here $\mathrm{adj}(I)$ is the adjoint  of $I$.
We generalize this result to integrally closed modules.

\begin{thm}
Let $(R,\mathfrak m)$ be a $2$-dimensional regular local ring with infinite residue field and $M$ be a finitely generated integrally closed torsion-free $R$-module of  rank $r$.
Then 
$$\mathrm{core}(M)=\mathrm{adj}(I(M))M,$$
where $I(M)$ is the ideal of $R$ generated by the $r$ minors of a representing matrix for $M$. 

\end{thm}

As an application of the theorem, we show the following proposition, which generalizes Corollary 3.13 in \cite{HS} to modules.

\begin{prop}
Let $(R,\mathfrak m)$ be a $2$-dimensional regular local ring with infinite residue field and $M,N$ be  finitely generated integrally closed torsion-free $R$-modules.
We assume that  $M\subset N$ and $M^{**}=N^{**}$, where $(-)^*$ denotes the functor $\mathrm{Hom}_R(-,R)$.
Then 
$\mathrm{core}(M)\subset \mathrm{core}(N)$.
\end{prop}

In Section $2$, we give necessary definitions and record various properties for later use.
In Section $3$, we determine the adjoint of the ideal of $R$ generated by the  minors of a representing matrix for  an integrally closed module $M$ over a $2$-dimensional regular local ring in terms of ideals of minors of any presentation of $M$ and also prove several properties of the adjoint.
In Section $4$, we prove the main theorem that relates the core of a module and the adjoint of an ideal of a $2$-dimensional regular local ring.
We also generalize several results  in \cite{HS} to integrally closed modules.

\begin{ack}
I am grateful to Prof. Shunsuke Takagi for the constant encouragements and many comments.
I also thank Futoshi Hayasaka for his helpful comments.
\end{ack}

\section{Preliminaries}
In this section, we give necessary definitions and record various properties for later use.

\noindent
{\bf \emph{Integral Closures, Reductions and Cores of  Modules}}

Let $R$ be a Noetherian domain and $K$ be its field of fractions.
Let $M$ be a finitely generated, torsion-free $R$-module.
By $M_K$ we denote the finite-dimensional $K$-vector space $M\otimes_R K$.
For any ring $S$ with $R\subset S \subset K$, we let $MS$ denote the $S$-submodule of $M_K$ generated by $M$. 
Let $S(M)$ denote the image of the symmetric algebra $\mathrm{Sym}^R(M)$ in the algebra $\mathrm{Sym}^K(M_K)$ under the canonical map.
Let $\mathrm{Sym}_n^K(M_K)$ (respectively $S_n(M)$) be the $n$th graded component of $\mathrm{Sym}^K(M_K)$ (respectively $S(M)$).

\begin{defn}
With notation as above, an element $v\in M_K$ is said to be integral over $M$ if $v\in MV$ for every discrete valuation ring $V$ of $K$ containing $R$.
The integral closure of $M$, denoted $\overline{M}$, is the set of all elements of $M_K$ that are integral over $M$.
The module $M$ is said to be integrally closed  if $M=\overline M$.
A submodule $N$ of $M$ is a reduction of $M$ if $M\subset \overline N$.
A reduction of $M$ is said to be minimal if it is minimal with respect to inclusion.
\end{defn}

\begin{thm}{\rm(Theorem 1.5 in \cite{R})}\label{Rees theorem1}
Let $R$ be a Noetherian domain and  let $M$ be a finitely generated torsion-free $R$-module of rank $r$.
For an element $v\in M_K$, the following are equivalent:
\begin{enumerate}
\item[{\rm(1)}]
The element $v$ is integral over $M$.

\item[{\rm(2)}]
 The element $v\in M_K=\mathrm{Sym}^K_1(M_K)$ is integral over $S(M)$.
\end{enumerate}

\end{thm}

\begin{thm}{\rm(Lemma 2.1 in \cite{R})}\label{Rees theorem2}
Let $R$ be a $d$-dimensional Noetherian  local domain with infinite residue field and  let $M$ be a non-free finitely generated torsion-free $R$-module of rank $r$.
Then $M$ has a minimal reduction which is generated by at most $r+d-1$ elements.
Further, a minimal generating set of a minimal reduction of $M$ forms part of a minimal generating set for $M$.
In particular, when $d=2$, $M$ has an $r+1$ generated minimal reduction.

\end{thm}

\begin{thm}{\rm(Theorem 5.2  in \cite{K})}\label{thm 5.2 in K}
Let $(R,\mathfrak m)$ be a $2$-dimensional regular  local ring with infinite residue field and  $M$ and $N$ be   finitely generated torsion-free integrally closed $R$-modules.
Then   $MN$ is integrally closed. 
In particular for an integrally closed ideal $\mathfrak a$ of $R$, $\mathfrak aM$ is integrally closed.
\end{thm}

\begin{defn}
Let $R$ be a  Noetherian   domain  and  let $M$ be a  finitely generated torsion-free $R$-module.
The core of $M$, denoted by $\mathrm{core}(M)$, is the intersection of all reductions of $M$.

\end{defn}

\noindent
{\bf \emph{Adjoints of  ideals}}

We will now review the definition of the adjoint of an ideal.

\begin{defn}
 Let $R$ be a regular domain with field of fractions $K$.
The adjoint of an ideal $I$ in $R$, denoted $\mathrm{adj}(I)$, is the ideal
$$\mathrm{adj}(I)=\bigcap_{v}\{r\in K|\ v(r)\ge v(I)-v(J_{R_v/R})\},$$
where the intersection varies over all  divisorial valuation with respect to $R$.
Here $R_v$ denotes the corresponding valuation ring for $v$ and $J_{R_v/R}$ denotes the Jacobian ideal of $R_v$ over $R$.   
\end{defn}

Here we recall some basic properties of adjoints of  ideals.

\begin{prop} {\rm(Lemma 18.1.2, Lemma 18.1.3 and Proposition 18.3.2 in \cite{HS book})}\label{basic properties}
Let $R$ be a regular  domain, $x$ an element in $R$ and $I, J$  ideals of $R$.
\begin{enumerate}
\item[{\rm(1)}]
$I\subset \overline I\subset \mathrm{adj}(I)=\mathrm{adj}(\overline I)$ and $\mathrm{adj}(I)$ is an integrally closed ideal of $R$.

\item[{\rm(2)}]
 $\mathrm{adj}(xI)=x\mathrm{adj}(I)$.

\item[{\rm(3)}]
If $I\subset J$, then $\mathrm{adj}(I)\subset \mathrm{adj}(J)$.

\item[{\rm(4)}] 
If $ J$ is minimally generated by $(x_1,\dots,x_l)$,  $\mathrm{ht}(J)=l$ and $R/J$ is regular,
then 
$$\mathrm{adj}(I)=\Biggl(\bigcap_{i=1}^l\frac{1}{x_i^{l-1}}\mathrm{adj}\Bigl(IR[\frac{J}{x_i}]\Bigr)\Biggr)\cap R.$$   
\end{enumerate}

\end{prop}

\begin{prop} {\rm(Proposition 1.3.1 in \cite{L})}\label{prop 1.3.1 in L}
Let $R$ be a regular domain,   $I$ be an  ideal of $R$ and 
$f:Y\to \mathrm{Spec}R$ be a proper birational morphism  such that  $Y$ has rational singularities (for example,  $Y$ regular) and $I\mathcal O_Y$ invertible.
Then
$$\mathrm{adj}(I)=H^0(Y,I\omega_Y).$$
Moreover if such a $f$ exists, then for any multiplicative system $M$ in $R$,
$$\mathrm{adj}(IR_M)=\mathrm{adj}(I)R_M.$$
\end{prop}

\begin{thm}{\rm(Theorem 3.14 in \cite{HS})}\label{core adj in HS}
Let $(R,\mathfrak m)$ be a $2$-dimensional regular local ring with infinite residue field and $\mathfrak a$ be an integrally closed $\mathfrak m$-primary ideal.
Then 
$$\mathrm{core}(\mathfrak a)=\mathrm{adj}(\mathfrak a)\mathfrak a=\mathrm{adj}(\mathfrak a^2).$$

\end{thm}

\noindent
{\bf \emph{Module Transforms  and Presenting matrices}}

We review the results of module transforms in \cite{K}.
Let $(R,\mathfrak m)$ be a $2$-dimensional regular local ring with infinite residue field and $M$ be a  finitely generated, torsion-free $R$-module.
Let $(-)^*$ denote the functor $\mathrm{Hom}_R(-,R)$.
The double dual $M^{**}$ of $M$ is a free $R$-module which canonically contains $M$ and  the quotient module $M^{**}/M$ is of finite length. Moreover, if $G$ is any free $R$-module containing $M$ such that $G/M$ is of finite length, then $M^{**}$ is isomorphic to $G$ up to unique isomorphism (Proposition 2.1 in \cite{K}).

Recall that the rank of $M$ is the vector space dimension of the $K$-vector space $M_K=M\otimes_RK$.
Let $\mathrm{rank}_R(M)$ denote the rank of $M$ and $\nu_R(M)$ denote the minimal number of generators of $M$.
We have $\mathrm{rank}_R(M)=\mathrm{rank}_R(M^{**})$ and $M\otimes_RK=M^{**}\otimes_RK$ since $M^{**}/M$ is of finite length.
Choose a basis for $M^{**}$ and a minimal generating set for $M$ and consider the matrix expressing this set of generators in terms of the chosen basis of $M^{**}$.
Considering the elements of $M^{**}$ as column vectors we get a $\mathrm{rank}_R(M)\times \nu_R(M)$ representing matrix for $M$.
The ideal of maximal minors, i.e., the minors of sized $\mathrm{rank}_R(M)$, is denoted $I(M)$ and is easily seen to be independent of the choices made.
We note that if $M$ is a free module then $I(M)=R$ and if $M$ is non-free then $I(M)$ is an $\mathfrak m$-primary ideal (See page 7 in \cite{M}).

We have the following exact sequence
$$0\to K\to G\to M^{**}\to M^{**}/M\to 0,$$
where $K$ and $G$ are free $R$-modules.
Note that the map $K\to G$ can be described by a $\mathrm{rank}_R(G)\times \mathrm{rank}_R(K)$ matrix $A$. This matrix $A$ is called a presenting matrix for $M$.
The ideal of $R$ generated by the $r$ minors of  $A$ is denoted $I_r(A)$.
We define $I_r(A)$ to be $R$ if $r\le 0$.

\begin{prop} {\rm(Proposition 2.5,  Proposition 4.3, Proposition 4.6,  Proposition 4.7 and Theorem 5.4 in \cite{K})}\label{properties of module in K}
Let $(R,\mathfrak m)$ be a $2$-dimensional  regular local ring with infinite residue field and $M$ be a finitely generated torsion-free $R$-module. 
Let  $x\in\mathfrak m\setminus \mathfrak m^2$ and  $T$ be a ring obtained by localization $R[\frac{\mathfrak m}{x}]$ at a maximal ideal containing $\mathfrak m R[\frac{\mathfrak m}{x}]$. 
\begin{enumerate}
\item[{\rm(1)}]
If $M$ is integrally closed, then  $I(M)$ is integrally closed, $MT$ is a integrally closed $T$-module and
$M=MR[\frac{\mathfrak m}{y}]\cap M^{**}$ for general $y\in \mathfrak m$.

\item[{\rm(2)}]
The following conditions are equivalent:
\begin{enumerate}
\item[{\rm(a)}]
There exists $y\in\mathfrak m\setminus \mathfrak m^2$ such that $M=MR[\frac{\mathfrak m}{y}]\cap M^{**}.$ 

\item[{\rm(b)}]
$\mathrm{ord_R}(I(M))=\nu_R(M)-\mathrm{rank}_R(M)$.
\end{enumerate}

\item[{\rm(3)}]
 $I(MT)=x^{-\mathrm{ord}_R(I(M))}I(M)T$.

\end{enumerate}
\end{prop}

\begin{prop} {\rm(Proposition 5.1 in \cite{K})}\label{intersection}
Let $(R,\mathfrak m)$ be a $2$-dimensional regular local ring with infinite residue field, $x$ be an element of  $\mathfrak m\setminus \mathfrak m^2$, $S=R[\frac{\mathfrak m}{x}]$ and $I$ be an ideal  of $R$.
Then 
$IS=S\cap \bigl(\bigcap IT\bigr)$ where the second intersection ranges over all  ring $T$ obtained by localization $S$ at a maximal ideal containing $\mathfrak mS$.
\end{prop}

\noindent
{\bf \emph{Buchsbaum-Rim multiplicity}}

We will review the definition of the Buchsbaum-Rim multiplicity for a module.
Let $R$ be a $d$-dimensional Noetherian local ring.
Let $P$ be a nonzero $R$-module of finite length with a free presentation
$$G\to F\to P\to 0.$$
Let $S(G)$ denote the image of $\mathrm{Sym}^R(G)$ in $\mathrm{Sym}^R(F)$.
Then $S(G)$ is a graded subring of   $\mathrm{Sym}^R(F)$ whose homogeneous components are denoted $S_n(G)$.
In \cite{BR}, Buchsbaum and Rim showed that the length $\ell_R(\mathrm{Sym}_n^R(F)/S_n(G))$ is asymptotically given by a polynomial function, $p(n)$, of $n$ of degree $\mathrm{rank}_R(F)+d-1$ and the leading coefficient of this polynomial is independent of the presentation chosen.

\begin{defn}
With notation as above, $(\mathrm{rank}_R(F)+d-1)!$ times leading coefficient of $p(n)$, denoted $e(P)$, is an invariant of $P$ and is called the Buchsbaum-Rim multiplicity of $P$.
We define the Buchsbaum-Rim multiplicity of the zero module to be $0$.
\end{defn}

\begin{prop} {\rm(Theorem 4.8 in \cite{K})}\label{buchsbaum rim}
Let $(R,\mathfrak m)$ be a $2$-dimensional  regular local ring with infinite residue field and $M$ be a non-free finitely generated torsion-free $R$-module. 
Let  $x\in\mathfrak m\setminus \mathfrak m^2$ and  $T$ be a ring obtained by localization $R[\frac{\mathfrak m}{x}]$ at a maximal ideal containing $\mathfrak m R[\frac{\mathfrak m}{x}]$. 
Then $$e((MT)^{**}/MT)<e(M^{**}/M).$$
\end{prop}

\section{The adjoint of an ideal}

In this section we study the  adjoint of the ideal of $R$ generated by the  minors of a representing matrix for an  integrally closed module  over a $2$-dimensional regular local ring.

\begin{thm}\label{Main}
Let $(R,\mathfrak m)$ be a $2$-dimensional regular local ring with infinite residue field and $M$ be a finitely generated integrally closed torsion-free $R$-module of  rank $r$.
Let $A$ be an $n\times n-r$ presenting matrix for $M$.
Then 
$$\mathrm{adj}(I(M))=I_{n-r-1}(A).$$
\end{thm}

\begin{proof}
If $M$ is a free module, then $I(M)=I_{n-r-1}(A)=R$ by the definition of the Fitting ideal.
Therefore we have that $\mathrm{adj}(I(M))=I_{n-r-1}(A)=R.$

We assume that $M$ is a non-free module.
We may assume that $n=\nu_R(M)$ (See the proof of Proposition 2.5 in \cite{M}).
We will prove that $\mathrm{adj}(I(M))=I_{n-r-1}(A)$ by induction on $e(M^{**}/M)$.
Note that $I_{n-r-1}(A)$ is integrally closed by Proposition 2.5 in \cite{M}.
By  Proposition \ref{properties of module in K}, we can choose $x,y\in \mathfrak m$  such that 
$$(x,y)=\mathfrak m,\ \  I(M)=I(M)R[\frac{x}{y}]\cap R=I(M)R[\frac{y}{x}]\cap R\ \ \mathrm{and}$$ 
$$I_{n-r-1}(A)=I_{n-r-1}(A)R[\frac{x}{y}]\cap R=I_{n-r-1}(A)R[\frac{y}{x}]\cap R.$$
Let $S=R[\frac{y}{x}]$ and $T=R[\frac{x}{y}]$.
Let $C_1$ be the set of the maximal ideals of $S$ containing $\mathfrak mS$ and $C_2$ be the set of the maximal ideals of $T$ containing $\mathfrak mT$.
By Proposition \ref{basic properties} and Proposition \ref{prop 1.3.1 in L}, we have
\begin{align*}
&\mathrm{adj}(I(M))\\
&=\frac{1}{x}\mathrm{adj}\Bigl(I(M)S\Bigr)\cap\frac{1}{y}\mathrm{adj}\Bigl(I(M)T\Bigr)\cap R\\
&=x^{n-r-1}\mathrm{adj}\Bigl(\frac{I(M)}{x^{n-r}}S\Bigr)\cap{y^{n-r-1}}\mathrm{adj}\Bigl(\frac{I(M)}{y^{n-r}}T\Bigr)\cap R\\
&={x}^{n-r-1}\bigcap_{\mathfrak n_1\in C_1}\mathrm{adj}\Bigl(\frac{I(M)}{x^{n-r}}S\Bigr)_\mathfrak {n_1}\cap {y}^{n-r-1}\bigcap_{\mathfrak n_2\in C_2}\mathrm{adj}\Bigl(\frac{I(M)}{y^{n-r}}T\Bigr)_{\mathfrak n_2}\cap R\\
&={x}^{n-r-1}\bigcap_{\mathfrak n_1\in C_1}\mathrm{adj}\Bigl(\frac{I(M)}{x^{n-r}}S_{\mathfrak n_1}\Bigr)\cap {y}^{n-r-1}\bigcap_{\mathfrak n_2\in C_2}\mathrm{adj}\Bigl(\frac{I(M)}{y^{n-r}}T_{\mathfrak n_2}\Bigr)\cap R.
\end{align*}
The third equality  holds since $\frac{I(M)}{x^{n-r}}S\supset (\mathfrak mS)^l$ and $\frac{I(M)}{y^{n-r}}S\supset (\mathfrak mT)^l$ for some natural number $l$.
By Proposition \ref{properties of module in K}, we have $\frac{I(M)}{x^{n-r}}S_{\mathfrak n_1}=I(MS_{\mathfrak n_1}),\frac{I(M)}{y^{n-r}}T_{\mathfrak n_2}=I(MT_{\mathfrak n_2})$ and $MS_{\mathfrak n_1},MT_{\mathfrak n_2}$ are integrally closed modules.
Therefore 
\begin{align*}
&{x}^{n-r-1}\bigcap_{\mathfrak n_1\in C_1}\mathrm{adj}\Bigl(\frac{I(M)}{x^{n-r}}S_{\mathfrak n_1}\Bigr)\cap {y}^{n-r-1}\bigcap_{\mathfrak n_2\in C_2}\mathrm{adj}\Bigl(\frac{I(M)}{y^{n-r}}T_{\mathfrak n_2}\Bigr)\cap R\\
&={x}^{n-r-1}\bigcap_{\mathfrak n_1\in C_1}\mathrm{adj}\Bigl(I(MS_{\mathfrak n_1})\Bigr)\cap {y}^{n-r-1}\bigcap_{\mathfrak n_2\in C_2}\mathrm{adj}\Bigl(I(MT_{\mathfrak n_2})\Bigr)\cap R.
\end{align*}
Let  $(A/x)_{\mathfrak n_1}$ denote the matrix with entries in $S_{\mathfrak n_1}$ obtained by dividing each entry of $A$ by $x$, where $\mathfrak n_1\in C_1$.
Then $(A/x)_{\mathfrak n_1}$ is an $n\times n-r$ presenting matrix for the finitely generated torsion-free  integrally closed module $MS_{\mathfrak n_1}$  (See the proof of Proposition 2.5 in \cite{M}).
By Proposition \ref{buchsbaum rim} and the induction hypothesis, 
$\mathrm{adj}\Bigl(I(MS_{\mathfrak n_1})\Bigr)=I_{n-r-1}\big((A/x)_{\mathfrak n_1}\big)$.
Therefore by   Proposition \ref{intersection},
\begin{align*}
{x}^{n-r-1}\bigcap_{\mathfrak n_1\in C_1}\mathrm{adj}\Bigl(I(MS_{\mathfrak n_1})\Bigr)\cap R&={x}^{n-r-1}\bigcap_{\mathfrak n_1\in C_1}I_{n-r-1}\big((A/x)_{\mathfrak n_1}\big)\cap R\\
&=\bigcap_{\mathfrak n_1\in C_1}I_{n-r-1}(A)S_{\mathfrak n_1}\cap R\\
&=I_{n-r-1}(A)S\cap R\\
&=I_{n-r-1}(A).
\end{align*}
In the same way as above, we have
\begin{align*}
{y}^{n-r-1}\bigcap_{\mathfrak n_2\in C_2}\mathrm{adj}\Bigl(I(MT_{\mathfrak n_2})\Bigr)\cap R
=I_{n-r-1}(A).
\end{align*}
Hence $$\mathrm{adj}(I(M))=I_{n-r-1}(A).$$

\end{proof}

\begin{cor}\label{adj=N:M}
Let $(R,\mathfrak m)$ be a $2$-dimensional regular local ring with infinite residue field.
Let $M$ be a finitely generated integrally closed torsion-free $R$-module and $N$ be any minimal reduction of $M$.
Then 
$$\mathrm{adj}(I(M))=(N:_RM).$$

\end{cor}

\begin{proof}
Let $r=\mathrm{rank}_R(M)$ and $A$ be an $n\times n-r$ presenting matrix for $M$.
Then $\mathrm{adj}(I(M))=I_{n-r-1}(A)$ by Theorem \ref{Main}. 
We have $I_{n-r-1}(A)=(N:_RM)$ by Corollary 2.6 in \cite{M}.
Therefore $\mathrm{adj}(I(M))=(N:_RM)$.
\end{proof}

Now we introduce some notation: $\mathrm{adj}^0(\mathfrak a)=\mathfrak a$ and, for $m\ge1$, $\mathrm{adj}^m(\mathfrak a)=\mathrm{adj}^{m-1}(\mathrm{adj}(\mathfrak a))$.

\begin{prop} {\rm(Proposition 3.16 in \cite{HS})}\label{adj^m in HS}
Let $(R,\mathfrak m)$ be a $2$-dimensional regular local ring with infinite residue field and $\mathfrak a$ be an $\mathfrak m$-primary integrally closed ideal.
Let $A$ be an $n\times n-1$ presenting matrix for $\mathfrak a$.
Then for $m\ge 0$,
$$\mathrm{adj}^m(\mathfrak a)=I_{n-1-m}(A).$$
\end{prop}

\begin{lem}  {\rm(Lemma 2.2 in \cite{M})}\label{B^t}
Let $(R,\mathfrak m)$ be a $2$-dimensional regular local ring with infinite residue field, $M$ be a finitely generated torsion-free non-free $R$-module of rank $r$ and $N$ be a minimal reduction of $M$.
Let ${\bf m}_1,\dots, {\bf m}_{n}$ be a set of minimal generators for $M$ such that  the first $r+1$ of these generators generate $N$.
Let $A$ be an $n\times n-r$ presenting matrix for $M$ with respect to  ${\bf m}_1,\dots, {\bf m}_{n}$ and $B$ be the $n-r-1\times n-r$ matrix obtained by deleting the first $r+1$ rows of $A$.
Then $B^t$ presents the ideal $(N:_RM)$, where $B^t$ is the transpose of $B$.
Furthermore, $I_{n-r-1}(B)=(N:_RM)$.
\end{lem}

\begin{lem}  \label{A=B}
Let $(R,\mathfrak m)$ be a $2$-dimensional regular local ring with infinite residue field, $M$ be a finitely generated integrally closed torsion-free non-free $R$-module of rank $r$ and $N$ be a minimal reduction of $M$.
Let ${\bf m}_1,\dots, {\bf m}_{n}$ be a set of minimal generators for $M$ such that  the first $r+1$ of these generators generate $N$.
Let $A$ be an $n\times n-r$ presenting matrix for $M$ with respect to  ${\bf m}_1,\dots, {\bf m}_{n}$ and $B$ be the $n-r-1\times n-r$ matrix obtained by deleting the first $r+1$ rows of $A$.
Then for $m\ge 1$, $$I_{n-r-m}(A)=I_{n-r-m}(B)=\mathrm{adj}^m(I(M)).$$

\end{lem}

\begin{proof}
By Theorem \ref{Main}, Corollary \ref{adj=N:M} and Lemma \ref{B^t},
$B^t$ presents $I_{n-r-1}(B)$ and 
$$I_{n-r-1}(A)=I_{n-r-1}(B)=\mathrm{adj}(I(M)).$$
Note that $I_{n-r-m}(B^t)=I_{n-r-m}(B)$.
By Proposition \ref{adj^m in HS}, we have for $m\ge 1$ $$\mathrm{adj}^m(I(M))=I_{n-r-m}(B).$$
Therefore $I_{n-r-m}(B)$ is an integrally closed ideal for $m\ge 1$ by Proposition \ref{basic properties}.
By  Proposition \ref{properties of module in K}, we can choose $x\in \mathfrak m\setminus \mathfrak m^2$  such that 
$$I_{n-r-m}(B)=I_{n-r-m}(B)R[\frac{\mathfrak m}{x}]\cap R.$$
Let $S=R[\frac{\mathfrak m}{x}]$.
Let $C$ be the set of the maximal ideals of $S$ containing $\mathfrak mS$. 
Let  $(A/x)_{\mathfrak n}$ (respectively $(B/x)_{\mathfrak n}$) denote the matrix with entries in $S_{\mathfrak n}$ obtained by dividing each entry of $A$ (respectively $B$) by $x$, where $\mathfrak n\in C$.
Then $$x^{n-r-m}\bigcap_{\mathfrak n\in C}I_{n-r-m}\big((B/x)_\mathfrak n\big)\cap S=\bigcap_{\mathfrak n\in C}I_{n-r-m}(B)S_\mathfrak n\cap S= I_{n-r-m}(B)S$$ by Proposition \ref{intersection} and 
$(A/x)_{\mathfrak n}$ is an $n\times n-r$ presenting matrix for the finitely generated torsion-free  integrally closed module $MS_{\mathfrak n}$  (See the proof of Proposition 2.5 in \cite{M}).
We will prove that $I_{n-r-m}(A)=I_{n-r-m}(B)$ by induction on $e(M^{**}/M)$.
By Proposition \ref{buchsbaum rim} and the induction hypothesis, we have
$$I_{n-r-m}\big((A/x)_\mathfrak n\big)=I_{n-r-m}\big((B/x)_\mathfrak n\big).$$
Hence
\begin{align*}
I_{n-r-m}(A)&\subset x^{n-r-m}\bigcap_{\mathfrak n\in C}I_{n-r-m}\big((A/x)_\mathfrak n\big)\cap R\\
&= x^{n-r-m}\bigcap_{\mathfrak n\in C}I_{n-r-m}\big((B/x)_\mathfrak n\big)\cap R\\
&= I_{n-r-m}(B)S\cap R\\
&=I_{n-r-m}(B)\\
&\subset I_{n-r-m}(A).
\end{align*}
Therefore $I_{n-r-m}(A)=I_{n-r-m}(B)$.
\end{proof}

Kodiyalam in \cite{K} raised the question of which all Fitting ideals of $M$ are integrally closed for an arbitrary integrally closed module $M$ over   a $2$-dimensional regular local ring with infinite residue field.
In the same paper, Kodiyalam proved that the first Fitting ideal of $M$ is integrally closed if $M$ is so.
We  obtain the following positive answer to Kodiyalam's question.
The following proposition is a generalization of Proposition 3.16 in \cite{HS}.
\begin{prop}\label{adj^m=I_{n-r-m}}
Let $(R,\mathfrak m)$ be a $2$-dimensional regular local ring with infinite residue field and $M$ be a finitely generated integrally closed torsion-free $R$-module of  rank $r$.
Let $A$ be an $n\times n-r$ presenting matrix for $M$.
Then for $m\ge 0$
$$\mathrm{adj}^m(I(M))=I_{n-r-m}(A).$$
In particular, $I_{n-r-m}(A)$ is an integrally closed ideal for $m\ge 0$.
\end{prop}

\begin{proof}
We have $I(M)=I_{n-r}(A)$ (See the proof of Proposition 2.2 in \cite{K}).
If $M$ is a free module, then $R=I(M)=\mathrm{adj}^m(I(M))=I_{n-r-m}(A)$ for any $m$.

We assume that $M$ is a non-free module.
By Theorem \ref{Rees theorem2}, we can choose
${\bf m}_1,\dots, {\bf m}_{n}\in M$ satisfying the assumption in Lemma \ref{A=B}. 
Therefore we have
 for $m\ge 1$, $I_{n-r-m}(A)=\mathrm{adj}^m(I(M))$ by Lemma \ref{A=B}.

\end{proof}

\begin{lem}\label{l(R/a)}
Let $(R,\mathfrak m)$ be a $2$-dimensional regular local ring with infinite residue field and $\mathfrak a$ be an $\mathfrak m$-primary integrally closed ideal.
Then
$$\ell(R/\mathfrak a)=e(\mathfrak a)-e(\mathrm{adj}(\mathfrak a))+e(\mathrm{adj}^2(\mathfrak a))-e(\mathrm{adj}^3(\mathfrak a))+\cdots.$$
\end{lem}

\begin{proof}
First note that $\ell(R/R)=e(R)=0$ and $\mathrm{adj}^i(\mathfrak a)=R$ for $i\gg 0$ by Proposition \ref{adj^m in HS}.

Let $I$ be a minimal reduction of $\mathfrak a$.
Then we have $\mathrm{adj}(\mathfrak a)=I:\mathfrak a$ by Proposition 3.3 in \cite{L}.
By Matlis duality, 
$$\ell(I:\mathfrak a/I)=\ell(\mathrm{Hom}_{R/I}(R/\mathfrak a,R/I))=\ell(R/\mathfrak a).$$
Therefore 
\begin{align*}
e(\mathfrak a)&=\ell(R/I)=\ell(R/I:\mathfrak a)+\ell(I:\mathfrak a/I)\\
&=\ell(R/\mathrm{adj}(\mathfrak a))+\ell(R/\mathfrak a).
\end{align*}
In the same way as above, we have
\begin{align*}
e(\mathrm{adj}^i(\mathfrak a))&=\ell(R/\mathrm{adj}^{i+1}(\mathfrak a))+\ell(R/\mathrm{adj}^i(\mathfrak a)).
\end{align*}
Hence 
$$\ell(R/\mathfrak a)=e(\mathfrak a)-e(\mathrm{adj}(\mathfrak a))+e(\mathrm{adj}^2(\mathfrak a))-e(\mathrm{adj}^3(\mathfrak a))+\cdots.$$
\end{proof}

We can calculate the colenth of the first Fitting ideal of an integrally closed module over a $2$-dimensional regular local ring using the Hilbert-Samuel multiplicities of the Fitting ideals of the module.
The following corollary is a generalization of Corollary 2.3 in \cite{CLU}.

\begin{prop}
Let $(R,\mathfrak m)$ be a $2$-dimensional regular local ring with infinite residue field and $M$ be a finitely generated integrally closed torsion-free $R$-module of  rank $r$.
Let $A$ be an $n\times n-r$ presenting matrix for $M$.
Then 
$$\ell(R/I_{n-r}(A))=e(I_{n-r}(A))-e(I_{n-r-1}(A))+\cdots+(-1)^{n-r-1}e(I_{1}(A)).$$
\end{prop}

\begin{proof}
This proposition follows from Proposition \ref{adj^m=I_{n-r-m}} and Lemma \ref{l(R/a)}.
\end{proof}

The integrally closed property is preserved by completion.

\begin{lem}\label{completion module}
Let $(R,\mathfrak m)$ be a $2$-dimensional regular local ring with infinite residue fields and 
$\widehat{R}$ be its completion.
Let $M$ be a finitely generated integrally closed torsion-free $R$-module of rank $r$.
Then $M\otimes_R\widehat{R}$ is a finitely generated integrally closed torsion-free $\widehat{R}$-module.
\end{lem}

\begin{proof}
Note that $M\subset M^{**}$ and $M^{**}$ is a free $R$-module of rank $r$.
Since $M\otimes_R\widehat{R}\subset M^{**}\otimes_R\widehat{R}\cong \widehat{R}^r$,  $M\otimes_R\widehat{R}$ is a  torsion-free $\widehat{R}$-module.

Note that $\overline{M\otimes_R\widehat{R}}\subset M^{**}\otimes_R\widehat{R}\cong \widehat{R}^r$.
Let $s=(s_1,\dots,s_r)\in \overline{M\otimes_R\widehat{R}}$, where $s_i\in \widehat{R}$.

By Theorem \ref{Rees theorem1}, there exist $a_i\in S_i(M\otimes_R\widehat{R})$ such that 
$$s^n+a_1s^{n-1}+\cdots+a_n=0.$$
Since $M^{**}/M$ is of finite length, we can choose $m\in \mathbb N$ such that $\mathfrak m^mM^{**}\subset M$.
We can regard $S(M^{**})$  as a subring of $S(M^{**}\otimes_R\widehat{R})$ since $S(M^{**})\cong R[z_1,\dots,z_r]$ and $S(M^{**}\otimes_R\widehat{R})\cong \widehat{R}[z_1,\dots,z_r]$. 
Choose $t_i\in R$ and $b_i\in S_i(M)$ such that $$s_i-t_i\in \mathfrak m^{mn}\widehat{R}\ \mathrm{and}\ a_i-b_i\in S_i(\mathfrak m^{mn}M^{**}\otimes_R\widehat{R}).$$
Let $t=(t_1,\dots,t_r)\in M^{**}=S_1(M^{**})$.
Then $$t^n+b_1t^{n-1}+\cdots+b_n\in S_n(\mathfrak m^{m}M^{**}\otimes_R\widehat{R})\cap S_n(M^{**})=S_n(\mathfrak m^{m}M^{**}).$$
By the choice of $m$, $S(\mathfrak m^{m}M^{**})\subset S(M)$, so that we can modify this equation to give an integral equation for $t$ over $S(M)$.
By Theorem \ref{Rees theorem1}, $t\in \overline{M}=M$.
Since $s-t\in \mathfrak m^{mn}M^{**}\otimes_R\widehat{R}\subset M\otimes_R\widehat{R}$,
it follows that $s\in M\otimes_R\widehat{R}.$
\end{proof}

\begin{cor}\label{adj completion}
Let $(R,\mathfrak m)$ and $(S,\mathfrak n)$ be  $2$-dimensional regular local rings with infinite residue fields, $f:R\to S$ be a local flat homomorphism and $M$ be a finitely generated integrally closed torsion-free $R$-module.
If  $M\otimes_RS$ is an integrally closed $S$-module, then 
$$\mathrm{adj}(I(M))S=\mathrm{adj}(I(M\otimes_RS)).$$
In particular, this holds if $S$ is the completion of $R$.
\end{cor}

\begin{proof}Let $r=\mathrm{rank}_R(M)$.
Note that  $M\otimes_RS\subset M^{**}\otimes_RS\cong {S}^r$, so that $M\otimes_RS$ is a torsion-free $S$-module.
Let $A=(a_{i,j})_{i,j}$ be a presenting matrix for $M$.
If $M$ is a free module, then $I(M)=R$ and $I(M\otimes_RS)=S$.
Hence $\mathrm{adj}(I(M))S=\mathrm{adj}(I(M\otimes_RS))=S$.
Suppose that  $M$ is a non-free module.
Let $F=M^{**}$. 
Then We have the following exact sequence
$$0\to K\to G\to F\to F/M\to 0,$$
where $K$ and $G$ are free $R$-modules.
Note that $F$ is a free module.
Since $S$ is a faithfully flat $R$-algebra, we have the following exact sequence
$$0\to K\otimes_RS\to G\otimes_RS\to F\otimes_RS\to (F\otimes_RS)/(M\otimes_RS)\to 0.$$
Therefore the matrix  $(f(a_{i,j}))_{i,j}$  is a presenting matrix for $M\otimes_RS$.
Hence we have $\mathrm{adj}(I(M))S=\mathrm{adj}(I(M\otimes_RS))$ by Theorem \ref{Main}.

The second part follows from  Lemma \ref{completion module}.
\end{proof}

%%%%%%%%%%%%%%%%%%%%%%%%%%%%%%%%%%%%%%%%%%%%%%%%%%%%%%%%%%%%%%%%%%%%%%%%%%%%%%%%%
\section{The Core of a Module}

In this section we study the  core of a finitely generated integrally closed torsion-free module over a $2$-dimensional regular local ring.

Let us now recall the main theorem in \cite{M}. 
Mohan determined the core of a finitely generated integrally closed torsion-free module over a $2$-dimensional regular local ring.
\begin{thm}\cite{M}\label{thm in M}
Let $(R,\mathfrak m)$ be a $2$-dimensional regular local ring with infinite residue field and $M$ be a finitely generated integrally closed torsion-free $R$-module of  rank $r$.
Let $A$ be an $n\times n-r$ presenting matrix for $M$.
Then 
$$\mathrm{core}(M)=I_{n-r-1}(A)M.$$
\end{thm}

The following theorem is a generalization of Theorem 3.14 in \cite{HS}.

\begin{thm}\label{core adj}
Let $(R,\mathfrak m)$ be a $2$-dimensional regular local ring with infinite residue field and $M$ be a finitely generated integrally closed torsion-free $R$-module of  rank $r$.
Then 
$$\mathrm{core}(M)=\mathrm{adj}(I(M))M.$$

\end{thm}

\begin{proof}
This theorem follows from Theorem \ref{Main} and Theorem \ref{thm in M}.
\end{proof}

\begin{cor}\label{core finite ideal}
Let $(R,\mathfrak m)$ be a $2$-dimensional regular local ring with infinite residue field.
Let  $M$ and $N$ be  finitely generated integrally closed torsion-free $R$-modules.
Then 
$$\mathrm{core}(M\oplus N)=\mathrm{adj}(I(M)I(N))(M\oplus N).$$
In particular for $\mathfrak m$-primary integrally closed ideals $\mathfrak a, \mathfrak b$ of $R$, 
$$\mathrm{core}(\mathfrak a\oplus \mathfrak b)=\mathrm{adj}(\mathfrak a\mathfrak b)\mathfrak a\oplus \mathrm{adj}(\mathfrak a\mathfrak b)\mathfrak b.$$
\end{cor}

\begin{proof}
Let  $\widetilde{M}$ (resp. $\widetilde{N}$) be a representing matrix for $M$ (resp. $N$).
Then 
$\begin{pmatrix}
\widetilde{M}&O\\
O&\widetilde{N}
\end{pmatrix}$
is a representing matrix for $M\oplus N$.
This implies that $I(M\oplus N)=I(M)I(N)$.
Therefore this corollary follows from Theorem \ref{core adj}.

Note that $I(\mathfrak a)=\mathfrak a$ for an $\mathfrak m$-primary ideal $\mathfrak a$.
Therefore the second part follows from the first part.
\end{proof}

The following corollary is a generalization of Corollary 3.13 in \cite{HS}.

\begin{cor}\label{core completion}
Let $(R,\mathfrak m)$ and $(S,\mathfrak n)$ be  $2$-dimensional regular local rings with infinite residue fields, $f:R\to S$ be a local flat homomorphism and $M$ be a finitely generated integrally closed torsion-free $R$-module.
If  $M\otimes_RS$ is an integrally closed $S$-module, then 
$$\mathrm{core}(M)\otimes_RS=\mathrm{core}(M\otimes_RS).$$
In particular, this holds if $S$ is the completion of $R$.
\end{cor}

\begin{proof}
By Corollary \ref{adj completion} and Theorem \ref{core adj},
\begin{align*}
\mathrm{core}(M\otimes_RS)&=\mathrm{adj}(I(M\otimes_RS))(M\otimes_RS)\\
&=\Bigl(\mathrm{adj}(I(M))S\Bigr)(M\otimes_RS)\\
&=\Bigl(\mathrm{adj}(I(M))M\Bigr)\otimes_RS\\
&=\mathrm{core}(M)\otimes_RS.
\end{align*}

The second part follows from  Lemma \ref{completion module}.
\end{proof}

The following proposition is a generalization of Proposition 3.15 in \cite{HS}.

\begin{prop}
Let $(R,\mathfrak m)$ be a $2$-dimensional regular local ring with infinite residue field.
Let   $M$ and $N$ be  finitely generated integrally closed torsion-free $R$-modules.
If  $M\subset N$ and $M^{**}=N^{**}$, then
$$\mathrm{core}(M)\subset \mathrm{core}(N).$$

\end{prop}

\begin{proof}
Note that   $\mathrm{rank}_RM=\mathrm{rank}_RN$  as $M^{**}=N^{**}$.
Since $M\subset N$ and $M^{**}=N^{**}$, we have $I(M)\subset I(N)$.
This implies that $\mathrm{adj}(I(M))\subset \mathrm{adj}(I(N))$ by Proposition \ref{basic properties}.
By Theorem \ref{core adj}, we have $\mathrm{core}(M)\subset \mathrm{core}(N).$
\end{proof}

\begin{rem}
In general $\mathrm{core}(\mathfrak a)$ is not  necessarily contained in $\mathrm{core}(\mathfrak b)$ for integrally closed ideals $\mathfrak a$ and $\mathfrak b$ with $\mathfrak a\subset \mathfrak b$.
Let $R=\mathbb C[[x,y]]$ and $\mathfrak m=(x,y)$.
Then $\mathrm{core}(\mathfrak m^2)=\mathfrak m^2\mathrm{adj}(\mathfrak m^2)=\mathfrak m^3$ and $\mathrm{core}((x^2))=(x^2)$.
Therefore $\mathrm{core}((x^2))\not\subset\mathrm{core}(\mathfrak m^2)$.
\end{rem}

$\ddot{\mathrm{H}}$ubl and  Swanson proved that the following powerful property of adjoints of ideals of a $2$-dimensional regular local ring.
\begin{prop}{\rm(See page 460 in \cite{HuS})}\label{subadditivity}
Let $(R,\mathfrak m)$ be a $2$-dimensional  regular local ring with infinite residue field and  $\mathfrak a,\mathfrak b$ be  ideals of $R$.
Then $$\mathrm{adj}(\mathfrak a\mathfrak b)\subset \mathrm{adj}(\mathfrak a)\mathrm{adj}(\mathfrak b).$$
\end{prop}

Lipman showed that $\mathrm{adj}(\mathfrak  a^{m+1})=\mathfrak a\mathrm{adj}(\mathfrak a^m)$ for a natural number $m$ and an ideal $\mathfrak a$ of a $2$-dimensional regular local ring (2.3 in  \cite{L}).
Our proof is just an imitation of the proof.
\begin{lem}\label{skoda}
Let $(R,\mathfrak m)$ be a $2$-dimensional  regular local ring with infinite residue field.
Let  $\mathfrak a$ and $\mathfrak b$ be  ideals of $R$.
Then for natural number $m$, $$\mathrm{adj}(\mathfrak a\mathfrak  b^{m+1})=\mathfrak b\mathrm{adj}(\mathfrak a\mathfrak b^{m}).$$
\end{lem}

\begin{proof}
If $\mathfrak b$ is a principal ideal, then this lemma holds by Proposition \ref{basic properties}.

Suppose that $\mathfrak b$ is an $\mathfrak m$-primary ideal.
Let $f:Y\to \mathrm{Spec}(R)$ be a morphism which factors as a sequence of blowups with nonsingular centers such that $\mathfrak a\mathcal O_Y$ and $\mathfrak b\mathcal O_Y$ are invertible.
Let $\mathfrak b_0=(a,b)$ be a reduction of $\mathfrak b$, so that $\mathfrak b_0\mathcal O_Y=\mathfrak b\mathcal O_Y$.
Let $F$ be the direct sum of $2$ copies of $(\mathfrak b\mathcal O_Y)^{-1}$.
Then we have the exact sequence
$$K(F,\sigma):0\to \Lambda^2 F\to F\to \mathcal O_Y \to 0,$$
where the map $\sigma: F\to \mathcal O_Y$ is defined by $(x,y)\mapsto ax+by$
(see page 111 in \cite{LT}).
Therefore $K(F,\sigma)\otimes \mathfrak a\mathfrak b^{m+1}\omega_Y$ is exact.
Note that  $H^1(Y,\mathfrak a\mathfrak b^{m-1}\omega_Y)=0$ (See 2.2 and Remark 2.2.1(b) in \cite{L}).
This implies that $$H^0(Y,\mathfrak a\mathfrak b^{m+1}\omega_Y)=\mathfrak b_0H^0(Y,\mathfrak a\mathfrak b^{m}\omega_Y).$$
By Proposition \ref{prop 1.3.1 in L}, $\mathrm{adj}(\mathfrak a\mathfrak  b^{m+1})=\mathfrak b_0\mathrm{adj}(\mathfrak a\mathfrak b^m).$
Since $$\mathfrak b_0\mathrm{adj}(\mathfrak a\mathfrak b^m)\subset \mathfrak b\mathrm{adj}(\mathfrak a\mathfrak b^m)\subset \mathrm{adj}(\mathfrak a\mathfrak b^{m+1})$$
by Proposition 18.2.1 in \cite{HS book},
we have $$\mathrm{adj}(\mathfrak a\mathfrak  b^{m+1})=\mathfrak b\mathrm{adj}(\mathfrak a\mathfrak b^{m}).$$
\end{proof}

\begin{lem}\label{lem I(M)}
Let $(R,\mathfrak m)$ be a $2$-dimensional  regular local ring with infinite residue field, $\mathfrak a$ be an  ideal of $R$ and $M$ be a  finitely generated  torsion-free $R$-module with rank $r$.
Then $I(\mathfrak aM)=I(\mathfrak a)^{r}I(M)$.
\end{lem}

\begin{proof}

If $\mathfrak a$ is a principal ideal, then  $I(\mathfrak a)=R$ and $I(\mathfrak aM)=I(M)$.
Therefore $I(\mathfrak aM)=I(\mathfrak a)^{r}I(M)$ if $\mathfrak a$ is a principal ideal.
Suppose that  $\mathfrak a$ is an $\mathfrak m$-primary ideal. 
Then $I(\mathfrak a)=\mathfrak a$ and $M^{**}$ is a free module such that  $M^{**}/\mathfrak aM$  is of finite length.
Therefore $(\mathfrak aM)^{**}=M^{**}$ by Proposition 2.1 in \cite{K}.
Let $x_1,\dots, x_m$ be generators of $\mathfrak a$ and $\widetilde{M}$ be a representing matrix for $M$.
Then the  $r\times m\nu_R(M)$ matrix $\bigl( x_1\widetilde{M}\ x_2\widetilde{M}\ \dots \ x_m\widetilde{M}\bigr)  $  is a representing matrix for $\mathfrak aM$, where $x_i\widetilde{M}$ is the matrix obtained by multiplying each entry of $\widetilde{M}$ by $x_i$.
Hence $I(\mathfrak aM)=\mathfrak a^{r}I(M)$.
Therefore $I(\mathfrak aM)=I(\mathfrak a)^{r}I(M)$.
\end{proof}

The following proposition is a generalization of Corollary 4.2.5 in \cite{S}.
\begin{prop}
Let $(R,\mathfrak m)$ be a $2$-dimensional  regular local ring with infinite residue field, $\mathfrak a$ be an integrally closed ideal of $R$ and $M$ be a finitely generated integrally closed torsion-free $R$-module with rank $r$.
Then
$$\mathrm{core}(\mathfrak aM)\subset I(\mathfrak a)^{r-1}\mathrm{core}(\mathfrak a)\mathrm{core}(M).$$
\end{prop}

\begin{proof}
By  Theorem \ref{core adj},  we have $\mathfrak a\mathrm{adj}(I(\mathfrak a))=\mathrm{core}(\mathfrak a).$
Note that $\mathfrak aM$ is integrally closed by Theorem \ref{thm 5.2 in K}.
By  Theorem \ref{core adj}, Proposition \ref{subadditivity}, Lemma \ref{skoda} and Lemma \ref{lem I(M)},
\begin{align*}
\mathrm{core}(\mathfrak aM)&=\mathrm{adj}(I(\mathfrak aM))\mathfrak aM\\
&=\mathrm{adj}(I(\mathfrak a)^{r}I(M))\mathfrak aM\\
&=I(\mathfrak a)^{r-1}\mathrm{adj}(I(\mathfrak a)I(M))\mathfrak aM\\
&\subset I(\mathfrak a)^{r-1}\mathrm{adj}(I(\mathfrak a))\mathrm{adj}(I(M))\mathfrak aM\\
&=I(\mathfrak a)^{r-1}\mathrm{core}(\mathfrak a)\mathrm{core}(M).
\end{align*}

\end{proof}

In \cite{HS}, Huneke and Swanson proved the following lemma in order to understand $\mathrm{core}\bigl(\mathrm{core}(\mathfrak a)\bigr)$.
\begin{lem} {\rm(Lemma 4.5 in \cite{HS})}\label{adj(core)=adj^2}
Let $(R,\mathfrak m)$ be a $2$-dimensional regular local ring with infinite residue field and $\mathfrak a$ be an $\mathfrak m$-primary integrally closed ideal of $R$.
Then 
$$\mathrm{adj}(\mathrm{core}(\mathfrak a))=\Bigl(\mathrm{adj}(\mathfrak a)\Bigr)^{2}.$$
\end{lem}

Now we want to understand  $\mathrm{core}\bigl(\mathrm{core}(M)\bigr)$.
We need the following lemma.
\begin{lem}\label{adj(core)=adj^r+1}
Let $(R,\mathfrak m)$ be a $2$-dimensional regular local ring with infinite residue field and $M$ be a finitely generated integrally closed torsion-free $R$-module of  rank $r$.
Then 
$$\mathrm{adj}(I(\mathrm{core}(M)))=\Bigl(\mathrm{adj}(I(M))\Bigr)^{r+1}.$$
\end{lem}

\begin{proof}
By Theorem \ref{core adj} and Lemma \ref{lem I(M)},
\begin{align*}
I(\mathrm{core}(M))=I\bigl(\mathrm{adj}(I(M))M\bigr)=\Bigl(I\bigl(\mathrm{adj}(I(M))\bigr)\Bigr)^rI(M).
\end{align*}
Since $I(M)$ is an $\mathfrak m$-primary ideal or $R$, $\mathrm{adj}(I(M))$ is an $\mathfrak m$-primary ideal or $R$ by Proposition \ref{basic properties}.
Therefore $I\bigl(\mathrm{adj}(I(M))\bigr)=\mathrm{adj}(I(M))$.
Note that $I(M)$ is an integrally closed ideal by Proposition \ref{properties of module in K}.
By Theorem \ref {core adj in HS},  Lemma \ref{skoda} and Lemma \ref{adj(core)=adj^2}, we have
\begin{align*}
\mathrm{adj}(I(\mathrm{core}(M)))&=\mathrm{adj}\Bigl(I(M)\bigl(\mathrm{adj}(I(M))\bigr)^r\Bigr)\\
&=\Bigl(\mathrm{adj}(I(M))\Big)^{r-1}\mathrm{adj}\Bigl(I(M)\mathrm{adj}(I(M))\Bigr)\\
&=\Bigl(\mathrm{adj}(I(M))\Big)^{r+1}.
\end{align*}

\end{proof}

Now we introduce some notation: $\mathrm{core}^1(M)=\mathrm{core}(M)$ and, for $n>1$, 
$\mathrm{core}^n(M)=\mathrm{core}^{n-1}(\mathrm{core}(M))$.

The following proposition is a generalization of Proposition 4.7 in \cite{HS}.
\begin{prop}
Let $(R,\mathfrak m)$ be a $2$-dimensional regular local ring with infinite residue field and $M$ be a finitely generated integrally closed torsion-free $R$-module of  rank $r$.
Then 
$$\mathrm{core}^n(M)=\Bigl(\mathrm{adj}(I(M))\Bigr)^{\frac{(r+1)^n}{r}-\frac{1}{r}}M.$$
In particular,
$$\mathrm{core}\bigl(\mathrm{core}(M)\bigr)=\Bigl(\mathrm{adj}(I(M))\Bigr)^{r+2}M.$$
\end{prop}

\begin{proof}
If $n=1$, this is just Theorem \ref{core adj}.
Now let $n>1$ and assume that the proposition holds for $n-1$.
Note that $\mathrm{core}(M)=\mathrm{adj}(I(M))M$ is integrally closed by Theorem \ref{thm 5.2 in K}.
Then by  hypothesis and Lemma \ref{adj(core)=adj^r+1}
\begin{align*}
\mathrm{core}^{n}(M)&=\mathrm{core}^{n-1}\bigl(\mathrm{core}(M)\bigr)\\
&=\Bigl(\mathrm{adj}\bigl(I(\mathrm{core}(M))\bigr)\Bigr)^{\frac{(r+1)^{n-1}}{r}-\frac{1}{r}}\mathrm{core}(M)\\
&=\Bigl(\mathrm{adj}(I(M))\Bigr)^{\frac{(r+1)^n}{r}-\frac{1}{r}}M.
\end{align*}

\end{proof}

%%%%%%%%%%%%%%%%
% bibliography
%%%%%%%%%%%%%%%

% Set bibliography items using the "thebibliography" environment  and following
% the style used by the AMS journals. 
%
% If the bibliography is generated by a bibtex database, use "amsplain" or
% "amsalpha" as bibliography style

\end{document}